\documentclass[reqno]{amsart}
\usepackage{amsfonts}
\usepackage[reqno]{amsmath}
\usepackage{amscd}
\usepackage{amssymb}
\usepackage{cite}
\usepackage{graphicx}

\providecommand{\U}[1]{\protect \rule{.1in}{.1in}}
\newtheorem{theorem}{Theorem}
\newtheorem{corollary}[theorem]{Corollary}
\newtheorem{lemma}[theorem]{Lemma}

\theoremstyle{definition}
\newtheorem{definition}[theorem]{Definition}

\theoremstyle{remark}
\newtheorem{remark}{Remark}

\begin{document}
\title[Finite Biorthogonal $M$ Matrix Polynomials ]{Finite Biorthogonal $M$
Matrix Polynomials}
\author{Esra G\"{U}LDO\u{G}AN LEKES\.{I}Z}
\address{\c{C}ankaya University, Faculty of Arts and Science, Department of
Mathematics, Ankara 06790, T\"{u}rkiye}
\email{esraguldoganlekesiz@cankaya.edu.tr}

\begin{abstract}
This paper provides a finite pair of biorthogonal matrix polynomials and
their finite biorthogonality, several
recurrence relations, matrix differential equation, generating function and integral representation.
\end{abstract}

\maketitle

\section{Introduction}

Orthogonal polynomials have been used for research and scientific studies in
many fields of mathematics, engineering and physics and have held an
important place in the scientific world for years.

The theory of orthogonal polynomials has been expanded over time and studied
in different forms. One of the extensions is the concept of the
orthogonality of two different polynomial families called biorthogonal,
while another extension of orthogonal polynomials is matrix orthogonal
polynomials. There are many studies in the literature on biorthogonal
polynomials \cite{GLfilo,Carlitz,VarmaSPhD} and matrix orthogonal
polynomials \cite%
{RS,DJL,AC,Cekim,DJ,DJ2,DJLP,Duran,DV,DL,DrauxJok,J,JCN,JCP,JC4,JDP,JDP2,SV,SV2,VarmaSPhD}%
, separately. While this is the case, there are not many works on
biorthogonal matrix polynomials. Those introduced so far are on families of
infinite polynomials. For instance, the pairs of biorthogonal Jacobi matrix polynomials and Konhauser matrix polynomials have been investigated in \cite{SBF,SF}.

In this study, we derive the biorthogonal matrix analogue of that we
defined in our previous work \cite{GLfilo}. Since there are some parametric
restrictions here, the defined family is called a finite biorthogonal matrix
polynomial set. In this way, the theory of biorthogonal matrix polynomials
is carried to a different dimension with the concept of "finite", which is
new for this field of study. This paper provides a wide and open field for
new research on this construction.

\bigskip

In the scalar case, the families \cite{MJ}
\begin{equation*}
M_{n}^{\left( h ,c \right) }\left( u\right) =\left( -1\right)
^{n}\Gamma \left( c+1+n\right) \sum_{l=0}^{n}\frac{\left( -n\right)
_{l}\left( n+1-h \right) _{l}\ u^{l}}{l!\Gamma \left( c+1+l\right) }
\end{equation*}
are finite orthogonal polynomials with respect to $w\left(u\right)=u^c \left(1+u\right)^{-\left(h+c\right)}$ over $\left[0,\infty \right)$ for $c
>-1$ and $h >1+2\max \left\{ n\right\}$.

Considering the self-adjoint variant for the differential equation
\begin{equation}
u\left(1+u\right) M_{n}^{\prime \prime }\left(u\right) +\left( 1+c-\left(
h-2 \right) u \right) M_{n}^{\prime }\left( u\right) -n\left(
n+1-h\right) M_{n}\left( u\right) =0,  \label{3}
\end{equation}
we write
\begin{equation*}
\int\limits_{0}^{\infty }M_{n}^{\left( h,c \right) }\left( u\right)
M_{s}^{\left( h,c \right) }\left(u\right)u^{c }\left(
1+u\right) ^{-\left( h +c \right) }du =\left\{
\genfrac{}{}{0pt}{}{\frac{n!\Gamma
\left( h-n\right) \Gamma \left(c+n+1\right) }{\left( h
-2n-1\right) \Gamma \left( h+c -n\right) }\delta _{n,s}  }{0}%
\genfrac{}{}{0pt}{}{,n=s}{,n \neq s}.
\right.
\end{equation*}
From this orthogonality relation, the following three term relation may be obtained:
\begin{eqnarray*}
&&  \left(n+1-h\right) \left( h-2n\right) M_{n+1}^{\left(h,c \right)
}\left(u\right)+\left( h-2-2n\right)_{3} u M_{n}^{\left( h,c \right) }\left( u\right)\\
&&-\left(h-2n-1\right)  
\left(h\left(
c +2n+1\right)-2n\left( 1+n\right) \right)
M_{n}^{\left( h,c \right) }\left( u\right)\\
&&=
n\left( c+n\right) \left( h-2\left(1+n\right)\right) \left(c+h-n\right) M_{n-1}^{\left( h,c \right) }\left( u\right).\\
\end{eqnarray*}

In 2024, G\"{u}ldo\u{g}an Lekesiz \cite{GLfilo} defined the pair of finite
biorthogonal polynomials suggested by the
finite orthogonal polynomials $M_{n}^{\left( h,c\right) }\left( u\right) $
as follows:%
\begin{equation} 
\genfrac{}{}{0pt}{}{M_{n}\left( h,c,\upsilon ;u\right) =\left( -1\right) ^{n}\left(
1+c\right) _{\upsilon n}\sum\limits_{m=0}^{n}\left( -1\right) ^{m}\binom{n}{m
}\frac{\left( n+1-h\right) _{\upsilon m}}{\left(1+c\right) _{\upsilon m}}%
\left( -u\right) ^{\upsilon m}, }{\mathfrak{M}_{s}\left(
 h,c,\upsilon ;u\right) =\sum\limits_{m=0}^{s}\sum\limits_{j=0}^{m}\left(
-1\right) ^{s+j}\binom{m}{j}\frac{\left(  h+c-s\right) _{m}}{m!}\left( \frac{
j+1+c}{\upsilon }\right) _{s}\ u^{m}\left( 1+u\right) ^{s-m},}
\label{BiortM}
\end{equation}%
where $h>1+N \left( 1+\upsilon\right) ,\ c>-1,N=\max \left\{ n,s\right\} $
and $\upsilon$ is a nonnegative integer. The pair satisfies the finite biorthogonality
relation%
\begin{equation*}
\int\limits_{0}^{\infty }u^{c}\left( u+1\right) ^{-\left( h+c\right)
}M_{n}\left(  h,c,\upsilon ;u\right) \mathfrak{M}_{s}\left(  h,c,\upsilon
;u\right) du=\left\{
\genfrac{}{}{0pt}{}{0\ \text{for}\ n,s=0,1,...;s\neq n }{not\ 0\text{
for}\ s=n\ \ \ \ \ \ \ \ \ \ \ \ \ \ \ \ \ \ \ \ \ \ \ \ \  }
\right.
\end{equation*}%
with respect to the weight function $w\left( u\right) =u^{c}\left(
1+u\right) ^{-\left( h+c\right) }$ over $\left( 0,\infty \right) $.

Then, in the scalar case, she introduced the finite orthogonal $M$ matrix
polynomials (foMp) \cite{GLsym} with the help of the finite orthogonal polynomials $%
M_{n}^{\left(  h,c\right) }\left( u\right) $ as follows.

For $n=0,1,2,...$, the eigenvalues $x$ and $z$ corresponding to the parameter matrices $H,C\in 
\mathbb{C}
^{p\times p}$,
satisfy the spectral conditions $\operatorname{Re}%
\left( z\right) >-1$ and $\operatorname{Re}\left( x\right) >2\max \left\{
n\right\} +1$ for $\forall z\in \Upsilon \left( C\right) $ and $\forall x\in \Upsilon \left( H\right) $.
Then, the foMp of degree $n$ is defined by%
\begin{eqnarray*}
M_{n}^{\left( H,C\right) }\left( u\right) &=&\sum\limits_{j=0}^{n}\left(
-1\right) ^{n}\dbinom{n}{j} \Gamma ^{-1}\left(
H-\left( n+j\right) I\right) \Gamma \left( H-nI\right)  \\
&&\times \left( -u\right) ^{j} \Gamma ^{-1}\left(
\left( 1+j\right) I+C\right) \Gamma \left( C+\left( 1+n\right) I\right) \\
&=&F\left(-nI, \left( 1+n\right) I-H;C+I;-u\right) \\
&&\times \left( -1\right) ^{n} \Gamma \left( \left( 1+n\right) I+C\right) \Gamma ^{-1}\left( I+C\right).
\end{eqnarray*}%

\bigskip

In this paper, inspired by the finite biorthogonal pair (\ref{BiortM}), we introduce a pair of finite biorthogonal matrix polynomials
related to the foMp defined in \cite%
{GLsym}. Section 2 includes some basic
notations and concepts on matrix polynomials. Section 3 and 4 present the main results. Also, the last section is the conclusion
part including a relationship between  the
biorthogonal $M$ matrix polynomials defined in third section and the biorthogonal Jacobi matrix polynomials \cite{SF} is given.

\section{Preliminaries}

Assume that $\Upsilon \left( S\right) $ is the set of all eigenvalues of any matrix $%
S \in 
\mathbb{C}
^{p\times p}$, where  $
\mathbb{C}
^{p\times p} $ is the real or complex matrix space of order $p$.
Let $T_{n}\left(u\right) $ be any real valued matrix polynomial and
defined as%
\begin{equation*}
T_{n}\left( u\right) =S_{n}u^{n}+S_{n-1}u^{n-1}+...+S_{1}u+S_{0},
\end{equation*}%
where $S_{j}\in
\mathbb{C}
^{p\times p}, \ 0\leq j\leq n$.

\begin{lemma}
\cite{DS} Suppose $S\in
\mathbb{C}
^{p\times p}$ for which $\Upsilon \left( S\right) \subset W ,$ where $W$ is an open set. Then,
\begin{equation}
h_{1}\left( S\right) h_{2}\left( S\right) =h_{2}\left( S\right) h_{1}\left( S\right)
\label{LA1}
\end{equation}
such that $h_{1}\left( z\right) $ and $h_{2}\left( z\right) $ are holomorphic
functions in $W$.
Therefore, if $SV=VS$, and$\ V\in 
\mathbb{C}
^{p\times p}$ is a matrix such that $\Upsilon \left( V\right) \subset W 
$, then%
\begin{equation}
h_{1}\left( S\right) h_{2}\left( V\right) =h_{2}\left( V\right) h_{1}\left( S\right) .
\label{LA2}
\end{equation}
\end{lemma}

\begin{definition}
The matrix version of the Pochhammer symbol is defined by%
\begin{equation}
\left( S\right) _{k}=S\left( S+I\right) \left( S+2I\right) ...\left(
S+\left( k-1\right) I\right) ,\ \ k\geq 1,  \label{poch}
\end{equation}%
where $S\in 
\mathbb{C}
^{p\times p}$, $\left( S\right) _{0}\equiv I$ and $I$ is the identity matrix.
\end{definition}

\begin{remark}
$\left( S\right) _{k}=\theta $ holds for $S=-jI, \ j=1,2,...$ and $k>j$.
\end{remark}

\begin{definition}
Let $Re(q )>0, \ \forall q \in \Upsilon \left( S\right)$ and $S\in 
\mathbb{C}
^{p\times p}$, $S$ is called a positive stable
matrix.
\end{definition}

\begin{definition}
The Gamma matrix function is defined by%
\begin{equation*}
\Gamma \left( S\right) =\int\limits_{0}^{\infty }u^{S-I}e^{-u}du
\end{equation*}
such that
\begin{equation*}
u^{S-I}=\exp \left( \left( S-I\right) \ln u\right),
\end{equation*}
and $S$ is a positive stable matrix.
Let $S+kI$ be invertible
for $k\geq 0$ and $S\in 
\mathbb{C}
^{p\times p}$. Then, by considering (\ref{LA1}) and (\ref{poch}), the equation%
\begin{equation}
\left( S\right) _{k}=\Gamma ^{-1}\left( S\right) \Gamma \left( S+kI\right)
,\ \ k\geq 1  \label{LA3}
\end{equation}
is satisfied 
\cite{JC2}.
\end{definition}

\begin{lemma}
Let $S\in 
\mathbb{C}
^{p\times p}$ be an arbitrary matrix in the light of (\ref{LA3}) and $D=\frac{d}{du}$. Then,
\begin{eqnarray*}
D^{r}\left( u^{S+kI}\right) &=&\left( \left( S+I\right) _{k-r}\right) ^{-1}\
\left( S+I\right) _{k}\ u^{S+\left( k-r\right) I} \\
&=&\Gamma \left( S+\left( k+1\right) I\right) \Gamma ^{-1}\left( S+\left(
k-r+1\right) I\right) u^{S+\left( k-r\right) I},\ \ r=0,1,2,... \textbf{.}
\end{eqnarray*}
\end{lemma}

\begin{definition}
The Beta matrix function is defined by \cite{JC2}
\begin{equation}
B(S,V)=\int\limits_{0}^{1}\left( 1-u\right) ^{V-I}u^{S-I}du,  \label{beta}
\end{equation}%
where $S,V\in 
\mathbb{C}
^{p\times p}$ are positive stable matrices.
\end{definition}

\begin{theorem}
If the matrices $S$,$V$ and $S+V$ are positive
stable such that $S,V\in 
\mathbb{C}
^{p\times p}$ are commutative, then
\begin{equation*}
B\left( S,V\right) =\Gamma
^{-1}\left( V+S\right) \Gamma \left( V\right) \Gamma \left( S\right) 
\end{equation*}%
exists \cite{JC}.
\end{theorem}

\begin{lemma}
Let the matrices $S,V\in 
\mathbb{C}
^{p\times p}$ satisfy the following conditions
\begin{equation*}
Re(w)>-1,\ Re(s)>2\max \left\{ n\right\} +1,\
\forall w\in \Upsilon (V),\ \forall s\in \Upsilon (S).
\end{equation*}
By (\ref{beta}),
\begin{equation*}
\int\limits_{0}^{\infty }u^{V}\left( 1+u\right) ^{-\left( S+V\right)
}du=B\left( S-I,V+I\right) \left( S+V\right) ^{-1}.
\end{equation*}
\end{lemma}

\begin{lemma}
\cite{H} The entire complex valued function $\Gamma ^{-1}\left( z\right)
=1/\Gamma \left( z\right) $ is the reciprocal scalar Gamma function. Then, the Riesz-Dunford functional calculus \cite{DS} satisfies that $%
\Gamma ^{-1}\left( S\right) $ is the inverse of $\Gamma \left( S\right) $ for any arbitrary matrix $S\in 
\mathbb{C}
^{p\times p}$ and well defined. If $S+kI$ has an inverse for $S\in 
\mathbb{C}
^{p\times p}, \ k=0,1,2,...$, then $\left( S\right) _{k}= \Gamma \left( kI+S\right) \Gamma
^{-1}\left(S\right)$.
\end{lemma}

\begin{lemma}
The matrix hypergeometric function denoted by
$F\left(S,V;K;z\right) $ has the following definition
\begin{equation}
F\left( S,V;K;z\right) =\sum\limits_{m\geq 0}\left(V\right) _{m} \left(
S\right) _{m} \left( \left( K\right)
_{m}\right) ^{-1} \frac{z^{m}}{m!} \label{2*}
\end{equation}%
such that $K+nI$ has an inverse for $n=0,1,...$, and $S$,$V$,$K\in 
\mathbb{C}
^{p\times p}$.
It converges for $\left\vert z\right\vert <1$ \cite{JC}.
\end{lemma}

\section{Main Results}

We define the following explicit representations
\begin{equation}
M_{n}^{\left( H,C\right) }\left( u;\upsilon \right)
=\sum\limits_{j=0}^{n}\left( -1\right) ^{j+n}\dbinom{n}{j}\left( 
\left(n+1\right) I-H\right) _{\upsilon j}\left( I+C\right) _{\upsilon n}\left(
I+C\right) _{\upsilon j}^{-1}\left( -u\right) ^{\upsilon j}  \label{1..}
\end{equation}%
and%
\begin{eqnarray}
\mathcal{M}_{n}^{\left( H,C\right) }\left( u;\upsilon \right)
&=&\sum\limits_{s=0}^{n}\sum\limits_{j=0}^{s}\frac{\left( -1\right) ^{j+n}}{%
s!}\dbinom{s}{j}\left( \frac{1}{\upsilon }\left( \left( j+1\right)
I+C\right) \right) _{n}  \label{2..} \\
&&\times \left( H+C-nI\right) _{s}\ u^{s}\left( 1+u\right) ^{n-s}  \notag
\end{eqnarray}
and call the pair as the finite biorthogonal M matrix polynomials,
where $\upsilon =1,2,...$, and matrices $H,C \in
\mathbb{C}
^{p\times p}$ satisfy the spectral conditions%
\begin{equation}
\operatorname{Re}\left(  z\right) >-1
,\ \operatorname{Re}\left( x\right) >1+\left( 1+\upsilon \right) \max \left\{ n\right\},\ \forall x\in \Upsilon \left( H\right) ,\
\forall z\in \Upsilon \left( C\right) ,\text{ and }HC=CH.  \label{SC}
\end{equation}

In fact, $M_{n}^{\left( H,C\right) }\left( y;\upsilon \right) $ has the
following hypergeometric form%
\begin{eqnarray}
M_{n}^{\left( H,C\right) }\left( u;\upsilon \right) &=&\ _{\upsilon
+1}F_{\upsilon }\left( -nI,\Delta \left( \upsilon ,\left( n+1\right)
I-H\right) ;\Delta \left( \upsilon ,C+I\right) ;\left( -u\right) ^{\upsilon
}\right)  \label{3..} \\
&&\times \left( -1\right) ^{n} \Gamma ^{-1}\left( I+C\right)\Gamma \left( C+\left( 1+\upsilon n\right)
I\right) ,  \notag
\end{eqnarray}%
where $\Delta \left( k,y \right) $ represents the set of $k$
parameters $\frac{y }{k},\frac{y+1}{k},...,\frac{y+k-1}{k}%
,\ k\geq 1.$

For $\upsilon =1$, (\ref{1..})-(\ref{3..}) get reduced to $M_{n}^{\left(
H,C\right) }\left( u\right) $, the foMp,
presented in \cite{GLsym}.

\begin{theorem}
Assume that $H,C\in 
\mathbb{C}
^{p\times p}$ are commutative such that $HC=CH$. Matrix polynomials $%
M_{n}^{\left( H,C\right) }\left( u;\upsilon \right) $ and $\mathcal{M}%
_{n}^{\left( H,C\right) }\left( u;\upsilon \right) $ satisfy the following
biorthogonality relation with the matrix weight function $W\left(
u,H,C\right) =u^{C}\left( 1+u\right) ^{-\left( C+H\right) }$ over $\left[ 0,\infty \right) $.%
\begin{eqnarray}
\Lambda _{ns} &=&\int\limits_{0}^{\infty }u^{C}\left( 1+u\right) ^{-\left(
H+C\right) }M_{n}^{\left( H,C\right) }\left( u;\upsilon \right) \mathcal{M}%
_{s}^{\left( H,C\right) }\left( u;\upsilon \right) du  \label{4..} \\
&=&\left\{
\genfrac{}{}{0pt}{}{s!\Gamma
^{-1}\left( H+C-sI\right) \Gamma \left( C+\left( \upsilon s+1\right) I\right) \Gamma \left(
H-sI\right) \left( H-I-\left( \upsilon +1\right) sI\right)
^{-1}}{0 \ \ \ \ \ \ \ \ \ \ \ \ \ \ \ \ \ \ \ \ \ \ \ \ \ \ \ \ \ \ \ \ \ \ \ \ \ \ \ \ \ \ \ \ \ \ \ \ \ \ \ \ \ \ \ \ \ \ \ \ \ \ \ \ \ \ \ \ \ \ \ \ \ \ \ \ \ \ \ \ \ \ \ \ \ \ \ }
\genfrac{}{}{0pt}{}{,s=n,}{,s \neq n.}
\right.  \notag
\end{eqnarray}
\end{theorem}

\begin{proof}
Replacing (\ref{1..}) and (\ref{3..}) by the integral in (\ref{4..}%
), we write
\begin{eqnarray}
\Lambda _{ns} &=&\left( -1\right) ^{n+s}\Gamma \left( H-\left( s+1\right)
I\right) \Gamma \left( \left( \upsilon n+1\right) I+C\right) \Gamma
^{-1}\left( H+C-sI\right)  \label{Bio} \\
&&\times \sum\limits_{j=0}^{n}\left( -1\right) ^{j}\dbinom{n}{j}\left(
-H+\left( n+1\right) I\right) _{\upsilon j}\left( -H+\left( s+2\right)
I\right) _{\upsilon j}^{-1}  \notag \\
&&\times \sum\limits_{m=0}^{s}\frac{\Gamma ^{-1}\left( C+\left( \upsilon
j+1\right) I\right) \Gamma \left( C+\left( m+\upsilon j+1\right) I\right) }{%
m!}  \notag \\
&&\times \sum\limits_{k=0}^{m}\left( -1\right) ^{k}\dbinom{m}{k}\left( \frac{%
1}{\upsilon }\left( C+\left( k+1\right) I\right) \right) _{s}  \notag
\end{eqnarray}
by using (\ref{LA1}), (\ref{LA2}) and (\ref{LA3}).

Assume that $g$ is a polynomial of degree $s$. Then, the equality \cite{Carlitz}
\begin{equation*}
f\left( u\right) =\sum\limits_{m=0}^{s}\dbinom{u}{m}\Delta ^{m}f\left(
0\right) ,\ \Delta ^{m}f\left( 0\right) =\sum\limits_{k=0}^{m}\left(
-1\right) ^{m-k}\dbinom{m}{k}f\left( k\right)
\end{equation*}%
or%
\begin{equation*}
f\left( u\right) =\sum\limits_{m=0}^{s}\frac{\left( -u\right) _{m}}{m!}%
\sum\limits_{k=0}^{m}\left( -1\right) ^{k}\dbinom{m}{k}f\left( k\right) .
\end{equation*}
By choosing the $s$-th degree matrix polynomials%
\begin{equation*}
f\left( u\right) =\left( \frac{1}{\upsilon }\left( C+\left( u+1\right)
I\right) \right) _{s},
\end{equation*}%
the equality (\ref{Bio}) leads to%
\begin{equation*}
\left( \frac{1}{\upsilon }\left( C+\left( u+1\right) I\right) \right)
_{s}=\sum\limits_{m=0}^{s}\frac{\left( -u\right) _{m}}{m!}%
\sum\limits_{k=0}^{m}\left( -1\right) ^{k}\dbinom{m}{k}\left( \frac{1}{%
\upsilon }\left( C+\left( k+1\right) I\right) \right) _{s}.
\end{equation*}%
For $u=-C-\left( \upsilon j+1\right) I$, in view of (\ref{LA1}) and (\ref%
{LA3}), we have%
\begin{eqnarray}
\left( -jI\right) _{s} &=&\sum\limits_{m=0}^{s}\frac{\Gamma ^{-1}\left(
\left( 1+\upsilon j\right) I+C\right) \Gamma \left( C+\left( m+1+\upsilon
j\right) I\right) }{m!}  \label{-jI} \\
&&\times \sum\limits_{k=0}^{m}\left( -1\right) ^{k}\dbinom{m}{k}\left( \frac{%
1}{\upsilon }\left( C+\left( k+1\right) I\right) \right) _{s}.  \notag
\end{eqnarray}%
Thereupon, we get%
\begin{eqnarray}
\Lambda _{ns} &=&\left( -1\right) ^{n+s}\Gamma \left( H-\left( 1+s\right)
I\right) \Gamma \left( C+\left( 1+\upsilon n\right) I\right) \Gamma
^{-1}\left( H+C-sI\right)  \label{5..} \\
&&\times \sum\limits_{j=0}^{n}\left( -1\right) ^{j}\dbinom{n}{j}\left(
-jI\right) _{s}\left( \left( 1+n\right)I-H\right) _{\upsilon j}\left(
-H+\left(2+s\right) I\right) _{\upsilon j}^{-1}.  \notag
\end{eqnarray}%
Under (\ref{LA3}),
\begin{equation*}
D^{k}\left[ u^{S+mI}\right] =\left( I+S\right) _{m}\left[ \left( I+S\right)
_{m-k}\right] ^{-1}u^{S+\left( m-k\right) I},\ \ m\geq 0
\end{equation*}%
holds for an arbitrary matrix $S\in 
\mathbb{C}
^{p\times p}$ and after some calculations, (\ref{5..}) becomes%
\begin{eqnarray*}
\Lambda _{ns} &=&\left( -1\right) ^{n+s+1}\Gamma \left( H-sI\right) \Gamma ^{-1}\left(
H+C-sI\right) \Gamma
\left( \left( \upsilon n+1\right) I+C\right) \left( I-H\right) _{s}\left( I-H\right) _{n}^{-1} \\
&&\times \sum\limits_{j=0}^{n-s}\left( -1\right) ^{j}s!\dbinom{n}{s}\dbinom{%
n-s}{j}\left( I-H\right) _{n+\upsilon \left( s+j\right) }\left( I-H\right)
_{s+\upsilon \left( j+s\right) +1}^{-1} \\
&=&\left( -1\right) ^{n+s+1}s!\dbinom{n}{s}\Gamma \left( H-sI\right) \Gamma
\left( C+\left( \upsilon n+1\right) I\right) \Gamma ^{-1}\left( H+C-sI\right)
\\
&&\times \left( I-H\right) _{s}\left( I-H\right) _{n}^{-1}\left(
D^{n-s-1}u^{-H+\left( n+\upsilon s\right) I}\sum\limits_{j=0}^{n-s}\left(
-1\right) ^{j}\dbinom{n-s}{j}u^{\upsilon j}\right) _{u=1} \\
&=&\left( -1\right) ^{n+s+1}s!\dbinom{n}{s} \Gamma
\left( C+\left( \upsilon n+1\right) I\right)  \Gamma \left( H-sI\right) \Gamma ^{-1}\left( H+C-sI\right)
\\
&&\times \left( I-H\right) _{s}\left( I-H\right) _{n}^{-1}\left(
D^{n-s-1}u^{-H+\left( n+\upsilon s\right) I}\left( 1-u^{\upsilon }\right)
^{n-s}\right) _{u=1}.
\end{eqnarray*}%
Therefore,
\begin{equation*}
\Lambda _{ns}=\left\{
\genfrac{}{}{0pt}{}{n!\Gamma \left(
-nI+H\right) \Gamma \left( \left( 1+\upsilon n\right) I+C\right) \Gamma
^{-1}\left( -nI+H+C\right) \left( H-\left(1+\left( 1+\upsilon \right) n\right)I\right)
^{-1}}{0 \ \ \ \ \ \ \ \ \ \ \ \ \ \ \ \ \ \ \ \ \ \ \ \ \ \ \ \ \ \ \ \ \ \ \ \ \ \ \ \ \ \ \ \ \ \ \ \ \ \ \ \ \ \ \ \ \ \ \ \ \ \ \ \ \ \ \ \ \ \ \ \ \ \ \ \ \ \ \ \ \ \ \ \ \ \ \ \ \ \ \ \ \ }
\genfrac{}{}{0pt}{}{,s=n,}{,s \neq n.}
\right.
\end{equation*}
When $\upsilon =1$, it is no coincidence that the result is the orthogonality for $M$ matrix polynomials $M_{n}^{\left( H,C\right)
}\left( u\right) $.
\end{proof}

\bigskip

Now, we show that the first set of finite biorthogonal $M$ matrix
polynomials $M_{n}^{\left( H,C\right) }\left( u;\upsilon \right) $ is
orthogonal with respect to $u$ basic polynomial of $\mathcal{M}_{n}^{\left(
H,C\right) }\left( u;\upsilon \right) $. It is hold that
\begin{equation}
\int\limits_{0}^{\infty }u^{C}\left( 1+u\right) ^{-\left( H+C\right)
}M_{n}^{\left(H,C\right) }\left( u;\upsilon \right) u^{i}du=\left\{
\genfrac{}{}{0pt}{}{0,\ i=0,1,...,n-1,}{\neq 0,\ i=n.\ \ \ \ \ \ \ \ \ \ \ \ \ \ \ \ \ \ }
\right. \label{M1}
\end{equation}
Replacing (\ref{1..}) in left-hand side of (\ref{M1}), we get
\begin{eqnarray*}
&&\int\limits_{0}^{\infty }u^{C}\left( 1+u\right) ^{-\left( H+C\right)
}M_{n}^{\left( H,C\right) }\left( u;\upsilon \right) u^{i}du \\
&=&\sum\limits_{j=0}^{n}\left( -1\right) ^{j+n}\dbinom{n}{j}\left( -H+\left(
1+n\right) I\right) _{\upsilon j}\left( I+C\right) _{\upsilon n}\left(
I+C\right) _{\upsilon j}^{-1}\left( -1\right) ^{\upsilon j} \\
&&\times \int\limits_{0}^{\infty }u^{C+i+\upsilon j}\left( 1+u\right)
^{-\left( H+C\right) }du \\
&=&\sum\limits_{j=0}^{n}\left( -1\right) ^{j+n}\dbinom{n}{j}\Gamma \left(
\left( 1+\upsilon n\right) I+C\right) \Gamma ^{-1}\left( \left( 1+\upsilon
j\right) I+C\right) \left( -H+\left( 1+n\right) I\right) _{\upsilon j} \\
&&\times \left( -1\right) ^{\upsilon j}\Gamma \left( \left(1+i+ \upsilon
j\right) I+C\right) \Gamma \left( H-\left( 1+i+\upsilon j\right) I\right)
\Gamma ^{-1}\left( H+C\right) .
\end{eqnarray*}
By using
\begin{equation*}
\frac{d^{i}}{du^{i}}\left[ u^{\left( \upsilon j+i\right) I+C}\right] \mid
_{u=1}=\Gamma \left( \left(1+i+ \upsilon j\right) I+C\right) \Gamma
^{-1}\left( \left( 1+\upsilon j\right) I+C\right) ,
\end{equation*}
we obtain
\begin{eqnarray*}
\int\limits_{0}^{\infty }u^{C}\left( 1+u\right) ^{-\left( H+C\right)
}M_{n}^{\left( H,C\right) }\left( u;\upsilon \right) u^{i}du\ \ \ \ \ \ \ \
\ \ \ \ \ \ \ \ \ \ \ \ \ \ \ \ \ \ \ \ \ \ \ \ \ \ \ \ && \\
=\sum\limits_{j=0}^{n}\left( -1\right) ^{n+\left( 1+\upsilon \right) j}
\dbinom{n}{j}\Gamma \left( \left(1+ \upsilon n\right) I+C\right) \left(
-H+\left(1+n\right) I\right) _{\upsilon j}\ \ \ \ \ \ \ \ \ \ && \\
\times \Gamma
^{-1}\left(C+H\right) \Gamma \left( H-\left(1+i+\upsilon j\right) I\right) \frac{d^{i}}{du^{i}}\left[ u^{\left( \upsilon
j+i\right) I+C}\right] \mid _{u=1} && \\
=\Gamma \left( H-nI\right) \Gamma \left( \left( \upsilon n+1\right)
I+C\right) \Gamma ^{-1}\left( H+C\right) \sum\limits_{j=0}^{n}\left( -1\right)
^{j+n}\dbinom{n}{j}\ \ \ \ \ \ \ \ \ && \\
\times \Gamma
\left( H-\left( i+1+\upsilon j\right) I\right) \Gamma ^{-1}\left( H-\left( \upsilon j+n\right) I\right) \frac{d^{i}}{du^{i}}\left[
u^{\left( \upsilon j+i\right) I+C}\right] \mid _{u=1} && \\
=\Gamma \left( \left( \upsilon n+1\right)
I+C\right) \Gamma ^{-1}\left( H+C\right) \Gamma \left( H-nI\right)  \sum\limits_{j=0}^{n}\left( -1\right)
^{1+i+j}\dbinom{n}{j}\ \ \ \ \ \ \ && \\
\times \frac{d^{n-\left( i+1\right) }}{du^{n-\left( i+1\right) }}\left[
u^{\left( n+\upsilon j\right) I-H}\right] \mid _{u=1}\frac{d^{i}}{du^{i}}%
\left[ u^{C+\left( i+\upsilon j\right) I}\right] \mid _{u=1} && \\
=\left\{
\genfrac{}{}{0pt}{}{0,\ 0\leq i<n,}{\neq 0,\ i=n.\ \ \ \ \ \ \ \ \ }
\right. \ \ \ \ \ \ \ \ \ \ \ \ \ \ \ \ \ \ \ \ \ \ \ \ \ \ \ \ \ \ \ \ \ \ \ \ \ \ \ \ \ \ \ \ \ \ \ \ \ \ \ \ \ \ \ \ \ \ \ \ \ &&
\end{eqnarray*}
Similarly, the second set of finite biorthogonal $M$ matrix polynomials $
\mathcal{M}_{n}^{\left( H,C\right) }\left( u;\upsilon \right) $ is
orthogonal with respect to $u^{\upsilon }$ basic polynomial of $
M_{n}^{\left( H,C\right) }\left( u;\upsilon \right) $. It is hold that
\begin{equation}
\int\limits_{0}^{\infty }u^{C}\left( 1+u\right) ^{-\left( H+C\right) }
\mathcal{M}_{n}^{\left(H,C\right) }\left( u;\upsilon \right)u^{\upsilon
i}du=\left\{
\genfrac{}{}{0pt}{}{0,\ i=0,1,...,n-1,}{\neq 0,\ i=n.\ \ \ \ \ \ \ \ \ \ \ \ \ \ \ \ \ \ }
\right.  \label{M2}
\end{equation}%
Substituting (\ref{2..}) in (\ref{M2}),
\begin{eqnarray*}
&&\int\limits_{0}^{\infty }u^{C}\left( 1+u\right) ^{-\left( H+C\right) }%
\mathcal{M}_{n}^{\left( H,C\right) }\left( u;\upsilon \right) u^{\upsilon
i}du \\
&=&\sum\limits_{m=0}^{n}\sum\limits_{s=0}^{m}\frac{\left( -1\right) ^{n+s}}{%
m!}\dbinom{m}{s}\left( \frac{1}{\upsilon }\left( C+\left( s+1\right)
I\right) \right) _{n}\left( H+C-nI\right) _{m} \\
&&\times \int\limits_{0}^{\infty }u^{\left( \upsilon i+m\right) I+C}\left(
1+u\right) ^{\left( n-m\right) I-\left( H+C\right) }du \\
&=&\frac{\left( -1\right) ^{n}\Gamma \left( -\left(1+n+\upsilon i\right)
I+H\right) \Gamma \left( \left( 1+\upsilon i\right) I+C\right) }{\Gamma \left(
H+C-nI\right) } \\
&&\times \sum\limits_{m=0}^{n}\frac{\left( C+\left( \upsilon i+1\right)
I\right) _{m}}{m!}\sum\limits_{s=0}^{m}\left( -1\right) ^{s}\dbinom{m}{s}%
\left( \frac{1}{\upsilon }\left( \left( 1+s\right) I+C\right) \right) _{n}.
\end{eqnarray*}
is obtained.
Considering (\ref{-jI}), we write%
\begin{equation*}
\int\limits_{0}^{\infty }u^{C}\left( 1+u\right) ^{-\left( H+C\right) }%
\mathcal{M}_{n}^{\left( H,C\right) }\left( u;\upsilon \right) u^{\upsilon
i}du=\left\{
\genfrac{}{}{0pt}{}{0,\ i=0,1,...,n-1,}{\neq 0,\ i=n.\ \ \ \ \ \ \ \ \ \ \ \ \ \ \ \ \ \ }
\right.
\end{equation*}

\section{Some Properties for the Finite Biorthogonal $M$ Matrix Polynomials}

We give matrix differential equation and obtain some generating functions
and recurrence relations for $M_{n}^{\left( H,C\right) }\left( u;\upsilon \right) $.
Along this part given that $H,C \in
\mathbb{C}
^{p\times p}$ are commutative.

\begin{theorem}
Polynomials $M_{n}^{\left( H,C\right) }\left( u;\upsilon \right) $ satisfy
the following differential equation%
\begin{equation}
\left[ uD\left( uD+C+\left( 1-\upsilon \right) I\right) _{\upsilon }-\left(
-u\right) ^{\upsilon }\left( uD-\upsilon n\right) \left( uD+\left(
n+1\right) I-H\right) _{\upsilon }\right] M_{n}^{\left( H,C\right) }\left(
u;\upsilon \right) =0.  \label{difequ}
\end{equation}
\end{theorem}

\begin{proof}
$M_{n}^{\left( H,C\right) }\left( u;\upsilon \right) $ are essentially $%
_{\upsilon +1}F_{\upsilon }$-type generalized matrix valued hypergeometric
functions, and the generalized hypergeometric function $_{m}F_{q}$ satisfies
the equation \cite{RS} of degree $\max \left\{ m,q\right\} $%
\begin{equation*}
w\left( w+V_{1}-1\right) \left( w+V_{2}-1\right) ...\left( w+V_{q}-1\right)
F\left( u\right) =u\left( w+S_{1}\right) \left( w+S_{2}\right) ...\left(
w+S_{m}\right) F\left( u\right) ,
\end{equation*}%
where $w=u\frac{\partial }{\partial u}$ is the differential operator and $%
F\left( u\right) $ is the $_{m}F_{q}$-type generalized matrix valued
hypergeometric function defined as%
\begin{equation*}
_{m}F_{q}
\left( \genfrac{}{}{0pt}{}{S_{1},...,S_{m}}{V_{1},...,V_{q}} ;u \right)
=\sum\limits_{j=0}^{\infty }\frac{u^{j}}{j!}\left(  \genfrac{}{}{0pt}{}{S_{1},...,S_{m}}{V_{1},...,V_{q}}\right) _{j}
\end{equation*}%
for $S_{1},...,S_{m},V_{1},...,V_{q}\in 
\mathbb{C}
^{p\times p}$ and%
\begin{equation*}
\left( \genfrac{}{}{0pt}{}{S_{1},...,S_{m}}{V_{1},...,V_{q}}\right) _{j+1}=\left(
V_{q}+j\right) ^{-1}...\left( V_{1}+j\right) ^{-1}\left(S_{1}+j\right)
...\left( S_{n}+j\right) \left( \genfrac{}{}{0pt}{}{S_{1},...,S_{m}}{V_{1},...,V_{q}}
\right) _{j}.
\end{equation*}
So, we have the differential equation (\ref{difequ}).
\end{proof}

\begin{remark}
In the scalar case $\upsilon =1$, (\ref{difequ}) arrives the matrix
differential equation for the foMp defined in \cite{GLsym}.
\end{remark}

\begin{theorem}
Polynomials given by (\ref{1..}) satisfy
the matrix generating functions%
\begin{eqnarray}
&&\sum\limits_{n=0}^{\infty }\left( \left( C+I\right) _{\upsilon n}\right) ^{-1} \left( I-H\right)
_{n}M_{n}^{\left(
H,C\right) }\left( u;\upsilon \right) \frac{\left(-t\right)^{n}}{n!}  \label{genfunc1} \\
&=&\left( 1-t\right) ^{H-I}\ _{1+\upsilon}F_{\upsilon }\left[ 
\genfrac{}{}{0pt}{}{\Delta \left( \upsilon+1,-H+I\right)}{\Delta \left( \upsilon ,C+I\right)};
\frac{t\left( 1+\upsilon \right) }{t-1}\left( \frac{u\left(1+\upsilon \right) }{\upsilon \left( t-1\right) }\right) ^{\upsilon }\right]  \notag
\end{eqnarray}%
and%
\begin{eqnarray}
\sum\limits_{n=0}^{\infty }\left( \left( I+C\right)
_{\upsilon n}\right) ^{-1}M_{n}^{\left( H+nI,C\right) }\left(u;\upsilon
\right) \frac{\left(-t\right)^{n}}{n!}
=e^{t}\ _{\upsilon }F_{\upsilon }\left[ 
\genfrac{}{}{0pt}{}{\Delta \left( \upsilon,I-H\right)}{\Delta \left( \upsilon ,C+I\right)};
-t\left( -u\right)^{\upsilon }\right] . \label{genfunc2}
\end{eqnarray}
\end{theorem}

\begin{proof}
Generating functions given in (\ref{genfunc1}) and (\ref{genfunc2}) can be arised by considering (\ref{LA1})-(\ref{LA3}) and Cauchy product.
\end{proof}

\begin{remark}
For $\upsilon =1$, (\ref{genfunc1}) is reduced to the generating functions introduced in \cite{GLsym} and 
(\ref{genfunc2}) is the new for the foMp $M_{n}^{\left(
H,C\right) }\left( u\right) $.
\end{remark}

\begin{theorem}
The matrix polynomials (\ref{1..}) hold the following matrix
recurrence relations%
\begin{equation}
uD\left( M_{n}^{\left( H,C\right) }\left( u;\upsilon \right) \right)
=\upsilon n\left[ M_{n}^{\left( H,C\right) }\left( u;\upsilon \right)
+\left( C+\left( \upsilon n-\upsilon +1\right) I\right) _{\upsilon }\
M_{n-1}^{\left( H-I,C\right) }\left( u;\upsilon \right) \right] ,\ \ n\geq 1,
\label{rec1}
\end{equation}%
\begin{equation}
uD\left( M_{n}^{\left( H,C\right) }\left( u;\upsilon \right) \right) =\left(
\upsilon nI+C\right) M_{n}^{\left( H,C-I\right) }\left( u;\upsilon \right)
-C M_{n}^{\left( H,C\right) }\left( u;\upsilon \right) ,  \label{rec2}
\end{equation}%
\begin{equation}
DM_{n}^{\left( H,C\right) }\left( u;\upsilon \right) =-n\upsilon \left(
-u\right) ^{\upsilon -1}\left( \left( 1+n\right) I-H\right) _{\upsilon }\
M_{n-1}^{\left( H-\left( 1+\upsilon \right) I,\upsilon I+C\right) }\left(
u;\upsilon \right) \ \ n\geq 1,  \label{rec3}
\end{equation}%
and, more generally,%
\begin{eqnarray}
D^{k}M_{n}^{\left( H,C\right) }\left( u;\upsilon \right) &=&\left( -\upsilon
\right) ^{k}\left( -u\right) ^{\left( \upsilon
-1\right)k } \left( n-k+1\right) _{k}\prod\limits_{j=0}^{k-1}\left( \left( 1+n+\upsilon j\right)
I-H\right) _{\upsilon }  \label{rec4} \\
&&\times \ M_{n-k}^{\left( H-k\left( \upsilon +1\right) I,C+k\upsilon
I\right) }\left( u;\upsilon \right) ,\ \ 0\leq k\leq n,  \notag
\end{eqnarray}%
where $D=\frac{d}{du}$.
\end{theorem}

\begin{proof}
Applying
\begin{equation*}
\left( I+S\right) _{\upsilon \left( n-1\right) }\left( \left( 1+\upsilon
\left(n-1\right) \right) I+S\right) _{\upsilon }=\left( I+S\right) _{\upsilon n}
\end{equation*}
to the right-hand side of (\ref{rec1}), and by (\ref{LA1}) and (\ref{LA2}), we arrive
\begin{eqnarray*}
&&\upsilon nM_{n}^{\left( H,C\right) }\left( u;\upsilon \right) +\upsilon
n\left( C+\left( \upsilon n-\upsilon +1\right) I\right) _{\upsilon }\
M_{n-1}^{\left( H-I,C\right) }\left( u;\upsilon \right) \\
&=&\upsilon n\sum\limits_{j=0}^{n}\left( -1\right) ^{j+n}\left[ \dbinom{n}{j}%
-\dbinom{n-1}{j}\right] \left( \left( 1+n\right) I-H\right) _{\upsilon
j}\left( I+C\right) _{\upsilon n}\left( \left( C+I\right) _{\upsilon
j}\right) ^{-1}\left( -u\right) ^{\upsilon j} \\
&=&u\sum\limits_{j=0}^{n}\left( -1\right) ^{n+j}\dbinom{n}{j}\left( \left(
1+n\right)I-H\right) _{\upsilon j}\left( I+C\right) _{\upsilon n}\left(
\left( C+I\right) _{\upsilon j}\right) ^{-1}\left( -\upsilon j\right) \left(
-u\right) ^{\upsilon j-1} \\
&=&uD\left( M_{n}^{\left( H,C\right) }\left( u;\upsilon \right) \right)
\end{eqnarray*}%
which gives recurrence relation (\ref{rec1}).

One can obtain recurrence relation (\ref{rec2}) by using a similar technique.

Taking the derivative of both sides of the recurrence relation with respect
to $u$, we have
\begin{eqnarray}
uD\left( M_{n}^{\left(H,C\right) }\left( u;\upsilon \right) \right)
&=&\upsilon n\left( -u\right) ^{\upsilon }\sum\limits_{j=1}^{n}\left(
-1\right) ^{n+j}\dbinom{n-1}{j-1}\left( -u\right) ^{ \left(
j-1\right) \upsilon }  \label{rec} \\
&&\times \left( \left( 1+n\right) I-H\right) _{\upsilon \left( j-1\right)
+\upsilon }\left( C+I\right) _{\upsilon n}\left( \left( C+I\right)
_{\upsilon \left( j-1\right) +\upsilon }\right) ^{-1}.  \notag
\end{eqnarray}
Using the fact
\begin{equation*}
\left( S\right) _{k}=\left( S\right) _{r}\left( S+rI\right) _{k-r},\ 0\leq
r\leq k
\end{equation*}%
and by (\ref{LA1}) and (\ref{LA2}), (\ref{rec}) results in
\begin{eqnarray*}
D\left( M_{n}^{\left( H,C\right) }\left( u;\upsilon \right) \right)
&=&-\upsilon n\left( -u\right) ^{\upsilon -1}\left( \left( n+1\right)
I-H\right) _{\upsilon }\sum\limits_{j=1}^{n-1}\left( -1\right) ^{n-1-j}
\dbinom{n-1}{j}\left( -u\right) ^{\upsilon j} \\
&&\times \left( nI-\left( H-\left( 1+\upsilon \right) I\right) \right)
_{\upsilon j}\left( I+C+\upsilon I \right) _{ \left( n-1\right) \upsilon }
\left( \left( I+C+\upsilon I\right) _{\upsilon j}\right) ^{-1}.
\end{eqnarray*}%
From the description of polynomials $%
M_{n}^{\left( H,C\right) }\left( u;\upsilon \right) $, we can release the recurrence relation (\ref{rec3}).

More generally, if the derivative with respect to $y$ is taken $k$ times by
considering (\ref{rec3}), then (\ref{rec4}) is obtained.
\end{proof}

\begin{remark}
Taking $\upsilon =1$ in Theorem 17, (\ref{rec2}),(\ref{rec3}) and (\ref{rec4}%
) reduce the matrix recurrence relations obtained in \cite{GLsym}, and (\ref%
{rec1}) appears to be new for the foMp.
\end{remark}

\begin{theorem}
Polynomials $M_{n}^{\left( H,C\right) }\left(
u;\upsilon \right) $ have the following integral representation%
\begin{eqnarray*}
\int\limits_{0}^{\left( -u\right) ^{\upsilon }}M_{n}^{\left( H,C\right)
}\left( u;\upsilon \right) u^{\upsilon -1}du &=&\frac{\upsilon \left( -1\right)
^{\upsilon }}{n+1}\left( \left(-H+ \left( n+1-\upsilon \right)
I\right) _{\upsilon }\right) ^{-1} \\
&&\times \left\{ \left( -1\right) ^{n}\left(C+ \left( 1-\upsilon \right)
I\right) _{\left( 1+n\right) \upsilon}+M_{n+1}^{\left( H+\upsilon
I,C-\upsilon I\right) }\left( u;\upsilon \right) \right\} .
\end{eqnarray*}
\end{theorem}

\begin{proof}
Considering (\ref{3..}) in the relation \cite[Eq. (2.14)]{Abdalla}, the
proof is completed.
\end{proof}

\bigskip

Similar to the scalar case, the following corollary can be given.

\begin{corollary}
Using the explicit representation (\ref{1..}) yields that%
\begin{eqnarray*}
&&M_{n}^{\left( -C-H,H\right) }\left( \frac{u-1}{2};\upsilon \right) =\left(
-1\right) ^{n}\Gamma \left( \left( 1+\upsilon n\right) I+H\right) \Gamma
^{-1}\left( H+I\right) \\
&&\ \ \ \ \ \ \ \ \ \ \ \ \ \ \ \ \ \times \ _{\upsilon +1}F_{\upsilon
}\left( -nI,\Delta \left( \upsilon ,C+H+\left( n+1\right) I\right) ;\Delta
\left( \upsilon ,I+H\right) ;\left( \frac{1-u}{2}\right) ^{\upsilon }\right)
\\
&&\ \ \ \ \ \ \ \ \ \ \ \ \ \ \ \ \ \ \ \ \ \ \ \ \ \ \ \ \ \ \ \ \ =\left(
-1\right) ^{n}n!J_{n}^{\left( H,C\right) }\left( u;\upsilon \right) .
\end{eqnarray*}%
Thus,%
\begin{eqnarray*}
&&J_{n}^{\left( H,C\right) }\left( u;\upsilon \right) =\frac{\left(
-1\right) ^{n}}{n!}M_{n}^{\left( -H-C,H\right) }\left( \frac{u-1}{2}%
;\upsilon \right) \\
&\Leftrightarrow &M_{n}^{\left( H,C\right) }\left( u;\upsilon \right)
=\left( -1\right) ^{n}n!J_{n}^{\left( C,-H-C\right) }\left( 2u+1;\upsilon
\right)
\end{eqnarray*}%
and%
\begin{eqnarray*}
&&K_{n}^{\left( H,C\right) }\left( u;\upsilon \right) =\frac{\left(
-1\right) ^{n}}{n!}\mathcal{M}_{n}^{\left( -H-C,H\right) }\left( \frac{u-1}{2
};\upsilon \right) \\
&\Leftrightarrow &\mathcal{M}_{n}^{\left( H,C\right) }\left( u;\upsilon
\right) =\left( -1\right) ^{n}n!K_{n}^{\left( C,-H-C\right) }\left(
2u+1;\upsilon \right) ,
\end{eqnarray*}%
where $J_{n}^{\left( H,C\right) }\left( u;\upsilon \right) $ and $%
K_{n}^{\left( H,C\right) }\left( u;\upsilon \right) $\ are the pair of
biorthogonal Jacobi matrix polynomial defined in \cite{DJL}.
\end{corollary}

\section{Conclusion}

The family constructed in this paper ensures
us various substantial applications for the finite biorthogonal $M$ matrix
polynomials. First, it is shown that this pair has the finite
biorthogonality condition (\ref{4..}) and satisfies the $\left( \upsilon
+1\right) $-order differential equation (\ref{difequ}). Also, some
generating functions, an integral
representation and  matrix recurrence relations for the finite biorthogonal $M$ matrix polynomials have been
presented.

\bigskip

$\mathbf{Declarations}$\\

$\mathbf{Ethical\ Approval}$
\\Not applicable.\\

$\mathbf{Funding }$
\\Not applicable.\\

$\mathbf{Availability\ of\ data\ and\ materials}$
\\Not applicable (data sharing not applicable to this article as
no data sets were generated or analyzed during the current study).\\

$\mathbf{Competing \ interests \ }$
\\The author declares that they have no competing interests.


\bigskip

\begin{thebibliography}{99}
\bibitem{GLfilo} G\"{u}ldo\u{g}an Lekesiz, E. Finite birthogonal polynomials
suggested by the finite orthogonal polynomials $M_{n}^{\left( p,q\right)
}\left( x\right) $. arXiv 2024, arXiv:2408.15010.

\bibitem{Carlitz} Carlitz, L. A note on certain biorthogonal polynomials.
Pac. J. Math. 1968, 24, 425-430.

\bibitem{VarmaSPhD} Varma, S. Some extensions of orthogonal polynomials. PhD
Thesis, 2013.

\bibitem{SV2} Sinap, A.; Van Assche, W. Orthogonal matrix polynomials and
applications. J. Comput. Appl. Math. 1996, 66(1--2), 27-52.

\bibitem{RS} Rom\'{a}n, P.; Simondi, S. The generalized matrix valued
hypergeometric equation. Int. J. Math. 2010, 21(2), 145-155.

\bibitem{DJL} Defez, E.; J\'{o}dar, L.; Law, A. Jacobi matrix differential
equation, polynomial solutions and their properties. Comput. Math. Appl.
2004, 48, 789-803.

\bibitem{AC} Alt\i n, A.; \c{C}ekim, B. Generating matrix functions for
Chebyshev matrix polynomials of the second kind. Hacet. J. Math. Stat. 2012,
41(1), 25-32.

\bibitem{Cekim} \c{C}ekim, B. New kinds of matrix polynomials. Miskolc Math.
2013, 14(3), 817-826.

\bibitem{DJ} Defez, E.; J\'{o}dar, L. Some applications of the Hermite
matrix polynomials series expansions. J. Comput. Appl. Math. 1998, 99,
105-117.

\bibitem{JCP} J\'{o}dar, L.; Company, R.; Ponsoda, E. Orthogonal matrix
polynomials and systems of second order differential equations. Differ. Equ.
Dyn. Syst. 1996, 3, 269-288.

\bibitem{DV} Duran, A.J.; Van Assche, W. Orthogonal matrix polynomials and
higher order recurrence relations. Linear Algebra Appl. 1995, 219, 261-280.

\bibitem{DJLP} Defez, E.; J\'{o}dar, L.; Law, A.; Ponsoda, E. Three-term
recurrences and matrix orthogonal polynomials. Util. Math. 2000, 57, 129-146.

\bibitem{DL} Duran, A.J.; L\'{o}pez-Rodriguez, P. Orthogonal matrix
polynomials: Zeros and Blumenthal's theorem. J. Approx. Theory 1996, 84,
96-118.

\bibitem{Duran} Duran, A.J. On orthogonal polynomials with respect to a
positive definite matrix of measures. Can. J. Math. 1995, 47, 88-112.

\bibitem{DJ2} Defez, E.; J\'{o}dar, L. Chebyshev matrix polynomials and
second order matrix differential equations. Util. Math. 2002, 61, 107-123.

\bibitem{J} James, A.T. Special functions of matrix and single argument in
statistics, In Theory and Applications of Special Functions, (Edited by R.A.
Askey), 497-520, Academic Press, 1975.

\bibitem{JCN} J\'{o}dar, L.; Company, R.; Navarro, E. Laguerre matrix
polynomials and system of second-order differential equations. Appl. Numer.
Math. 1994, 15, 53--63.

\bibitem{DrauxJok} Draux, A.; Jokung-Nguena, O. Orthogonal polynomials in a
non-commutative algebra. Non-normal case. IMACS Annals of Computing and
Applied Mathematics 1991, 9, 237-242.

\bibitem{JC4} J\'{o}dar, L.; Company, R. Hermite matrix polynomials and
second order matrix differential equations. J. Approx. Theory Appl. 1996,
12(2), 20-30.

\bibitem{JDP} J\'{o}dar, L.; Defez, E.; Ponsoda, E. Matrix quadrature and
orthogonal matrix polynomials. Congressus Numerantium 1995, 106, 141-153.

\bibitem{SV} Sinap, A.; Van Assche, W. Polynomial interpolation and Gaussian
quadrature for matrix valued functions. Linear Algebra Appl. 1994, 207,
71-114.

\bibitem{JDP2} J\'{o}dar, L.; Defez, E.; Ponsoda, E. Orthogonal matrix
polynomials with respect to linear matrix moment functionals: Theory and
applications. J. Approx. Theory Appl. 1996, 12(1), 96--115.

\bibitem{SF} Varma, S.; Ta\c{s}delen, F. Biorthogonal matrix polynomials
related to Jacobi matrix polynomials. Comput. Math. Appl. 2011, 62(10),
3663-3668.

\bibitem{SBF} Varma, S.; Cekim, B.; Tasdelen Yesildal, F. On Konhauser
Matrix Polynomials. Ars Comb. 2011, 100, 193-204.

\bibitem{MJ} Masjed-Jamei, M. Three finite classes of hypergeometric
orthogonal polynomials and their application in functions approximation.
Integral Transforms Spec. Funct. 2002, 13(2), 169-190.

\bibitem{GLsym} G\"{u}ldo\u{g}an Lekesiz, E. Finite Orthogonal $M$ 
Matrix Polynomials. Symmetry 2025, 17(7), 996.

\bibitem{DS} Dunford, N.; Schwartz, J. Linear Operators; 
Interscience: New York, NY, USA, 1963; Volume I.

\bibitem{JC2} J\'{o}dar, L.; Cort\'{e}s, J.C. Some properties of Gamma and
Beta matrix functions. Appl. Math. Lett. 1998, 11(1), 89-93.

\bibitem{JC} J\'{o}dar, L.; Cort\'{e}s, J.C. On the hypergeometric 
matrix function. J. Comput. Appl. Math. 1998, 99, 205-217.

\bibitem{H} Hille, E. Lectures on Ordinary Differential Equations;
Addison-Wesley: New York, NY, USA 1969.

\bibitem{Abdalla} Abdalla, M. Further results on the generalised
hypergeometric matrix functions. Int. J. Comput. Sci. Math. 2019, 10(1).

\end{thebibliography}
\end{document}